\documentclass[12pt,a4paper,twoside,reqno]{amsart} 
\usepackage{amsfonts, amsthm, amsmath, amssymb}
\usepackage{hyperref}
\hypersetup{colorlinks=false}

\usepackage[margin=2.9cm]{geometry}

\usepackage{helvet}

\usepackage{graphicx}
\usepackage{amsmath}

\usepackage{amsthm}
\newtheorem{theorem}{Theorem}
\newtheorem{lemma}{Lemma}[section]

\newtheorem*{ack}{Acknowledgement}

\begin{document}
	
	\author{Sumit Kumar, Ritabrata Munshi, and Saurabh Kumar Singh}
	\title{ Sub-convexity bound for  $GL(3) \times GL(2)$ $L$-functions: Hybrid level aspect}

	\address{Sumit Kumar \newline  Alfr\'ed R\'enyi institute of Mathematics, Budapest, Re\'altanoda utca 13-15, 1053.}
	\email{sumitve95@gmail.com}
	\address{Ritabrata Munshi \newline  Stat-Math Unit, Indian Statistical Institute, 203 B.T. Road, Kolkata 700108, India.}
	\email{ritabratamunshi@gmail.com}
	
	\address{Saurabh Kumar Singh \newline Department of Mathematics and Statistics, Indian Institute of Technology, Kanpur-208016, India.}
	\email{skumar.bhu12@gmail.com}
	
	\date{\today}
	
	.
	\maketitle
	
	\begin{abstract}
		Let $F$ be a $G L(3)$ Hecke-Maass cusp form of prime level $P_1$ and let $f$ be a $G L(2)$ Hecke-Maass cuspform of prime level $P_2$. 	In this article, we will prove a subconvex bound for the $G L(3) \times G L(2)$  Rankin-Selberg $L$-function $L(s,F\times f)$ in the level aspect for certain ranges of the parameters $P_1$ and $P_2$.
	\end{abstract}

	\section{Introduction}

		In this paper we continue our study of the subconvexity problem for the degree six $GL(3)\times GL(2)$ Rankin-Selberg $L$-functions using the delta symbol approach \cite{ICM}. In the first paper \cite{munshi12} on this theme the second author established subconvex bounds in the $t$-aspect for these $L$-functions. Since then the method has been extended by the first author and the third author together with Sharma and Mallesham (see \cite{sumit}, \cite{KMS}, \cite{KMS2}, \cite{prahlad}), to produce various instances of subconvexity in the spectral aspect and twist aspect. Indeed the delta symbol approach has worked quite well in the $t$-aspect and the spectral aspect. However its effectivity and adaptability in the more arithmetic problem of level aspect remains a point of deliberation.  
	In particular, it seems that new inputs are required to tackle the level aspect problem for such $L$-functions, especially when one of the forms is kept fixed and the level of the other varies. However, as was shown in the lower rank case of Rankin-Selberg convolution of two $GL(2)$ forms \cite{HM}, the problem can be more tractable when both the forms vary in certain relative range. The aim of the present paper is to prove such a result for $GL(3)\times GL(2)$ Rankin-Selberg convolution.

%
	\begin{theorem}\label{main theorem}
	Let $P_1$ and $P_2$ be two distinct primes. Let $F$ be a Hecke-Maass cusp form for the congruence subgroup $\Gamma_0(P_1)$ of $S L(3,\mathbb{Z})$ with trivial nebentypus. Let $f$ be a holomorphic or Maass cusp form for the congruence subgroup $\Gamma_0(P_2)$ of $SL(2,\mathbb{Z})$ with trivial nebentypus. Let $ \mathcal{Q} = P_1^2 P_2^3 $ be the arithmetic conductor of the Rankin-Selberg convolution of the above two forms. Then we have 
	\begin{align*}
		L(1/2,F\times f) \ll \mathcal{Q}^{1/4+\varepsilon}\left( \frac{P_1^{1/4}}{P_2^{3/8}}+\frac{P_2^{1/8}}{P_1^{1/4}}\right).
	\end{align*}
\end{theorem}
%
%
	Note that the convexity bound is given by $\mathcal{Q}^{1/4+\varepsilon}$. Thus  the above bound is subconvex in the range  
	$$
	P_2^{1/2+\epsilon} <P_1 <P_2^{3/2-\epsilon}.
	$$
	This provides the first instance of a subconvex bound in the level aspect for a degree six $L$-function which is not a character twist of a fixed $L$-function.	The bound is strongest when $P_1$ and $P_2$ are roughly of same size $P_1\approx P_2$, in which  case Theorem~\ref{main theorem} gives 
\begin{align*}
	L(1/2,F\times f) \ll \mathcal{Q}^{1/4-1/40+\varepsilon}.
\end{align*}
The exponent $1/4-1/40$ appears in other contexts as well and it seems to be the limit of the delta symbol approach.   We also note that our proof with some suitable modifications works even in the case of composite levels $P_1$ and $P_2$. But to keep the exposition  simple and clean we will only give full details for the case of prime levels. 
\\

	For a detailed introduction to automorphic forms on higher rank groups and for basic analytic properties of Rankin-Selberg convolution $L$-functions we refer the readers to Goldfeld's book \cite{DG}. Our treatment will be at the level of $L$-functions, and the Voronoi summation formulae for $GL(2)$ and $GL(3)$ are the only input that we need from the theory of automorphic forms. For a broader introduction to the subconvexity problem and its applications we refer the readers to \cite{michel-pc}, \cite{ICM}.
	\\
	
	Historically the level aspect subconvexity problem has proved to be more challenging compared to the spectral aspect or the $t$-aspect, regardless of the method adopted. Indeed Weyl shift is all one needs to prove the $t$-aspect subconvexity for $\zeta(s)$ (cf. \cite{Weyl}). Whereas Burgess had to nontrivially extend Weyl's ideas and had to invoke Riemann Hypothesis for curves over finite fields, to obtain the first level aspect subconvexity result $L(1/2,\chi)\ll q^{3/16+\varepsilon}$ (cf. \cite{Burgess}).  In 1990s Duke, Friedlander and Iwaniec (\cite{DFI},\cite{DFI-2},\cite{DFI-2.1}) used the amplification technique to obtain the level aspect subconvexity for $GL(2)$ $L$-functions. The amplification method was extended by Kowalski, Michel and Vanderkam \cite{KMV} to Rankin-Selberg convolutions $GL(2)\times GL(2)$. In \cite{venkatesh} Venkatesh used ergodic theory to study orbital integrals, and thus obtained level aspect subconvex bounds for triple products $GL(2)\times GL(2)\times GL(2)$, where two forms are fixed and one varies. A similar technique was also adopted by Michel-Venkatesh \cite{MV} for $GL(2)\times GL(2)$ $L$-functions over any number fields. The level aspect subconvexity problem for any genuine $GL(d)$ $L$-function with $d>2$ remains an important open problem.
	\\
	
	Our interest in the subconvexity problem for $GL(3)\times GL(2)$ Rankin-Selberg convolution is kindled by two factors. First there is a structural advantage which makes the $GL(n)\times GL(n-1)$ $L$-functions a suitable candidate for analytic number theoretic exploration. Indeed the case of $n=2$ has been extensively studied in the literature, as we will see below, and we want to extend to the next level $n=3$. Secondly, $GL(3)\times GL(2)$ Rankin-Selberg convolutions appear in important applications, like the Quantum Unique Ergodicity, and so it is important to analyse different aspects of the subconvexity problem for these $L$-functions with the aim of developing techniques that will eventually work in the required scenarios, e.g. spectral aspect subconvexity for symmetric square $L$-functions. Finally, let us also stress, that we are motivated to explore the scope of the delta symbol approach to subconvexity and other related problems. After initial success of the second author \cite{ICM}, the method has been extended, simplified and generalised by several researchers, e.g. see \cite{HN}, \cite{aggarwal}, \cite{AHLQ}, \cite{AHLS}, \cite{KLMS}, \cite{sumit}, \cite{MS}, \cite{prahlad} and \cite{LMS}.    
	\\

	
	The twists of $GL(2)$ $L$-functions by Dirichlet characters, or in other words $GL(2)\times GL(1)$ $L$-functions have been studied extensively in the literature, ever since the breakthrough work of Duke, Friedlander and Iwaniec \cite{DFI}. Hybrid subconvexity have also been studied for these $L$-functions. Since this is the lower rank analogue of the $L$-function we are investigating in this paper, we briefly recall some results in this basic case. Let $ f$ be a $GL(2)$ new form of level $P_2$ and let $ \chi$ be a primitive Dirichlet character of modulus $ P_3$. Suppose $(P_2,P_3)=1$, then $ \mathcal{Q} = P_2 P_3^2$ is the arithmetic conductor of $ L (1/2, f \otimes \chi)$. Different methods are now available to prove hybrid sub-convexity bound, when the levels of forms vary in a relative range, say $ P_2 \sim P_3^\eta$. Blomer and Harcos \cite{BH} used amplification technique to  prove 
	\begin{align*}
		L (1/2, f \otimes \chi) \ll  \mathcal{Q}^{\frac{1}{4} + \epsilon} \left( {Q}^{- \frac{1}{8 (2 + \eta)} } + {Q}^{- \frac{1 - \eta}{4 (2 + \eta)} } \right) 
	\end{align*} for $ 0 < \eta < 1$.   Aggarwal,  Jo and  Nowland \cite{KJN} used classical delta method to prove  
	\begin{align*}
		L (1/2, f \otimes \chi) \ll  \mathcal{Q}^{\frac{1}{4}  - \frac{ 2 - 5 \eta}{20 (2 + \eta)}+  \epsilon} 
	\end{align*} for $ 0 < \eta < 2/5$.  Computing the average of the second moment of $ L (1/2, f \otimes \chi)$ over a family of forms, Hou and Chen \cite{HC} extended the range of $\eta$ to $ 0< \eta < 3/2 - \theta$, where $ \theta$ is any admissible exponent towards the Petersson-Ramanujan conjecture for the Fourier coefficients.  Currently, the result of Hou and Chen  yields the widest range  $ P_2 \ll P_3^{3/2 - \delta}$, but it falls short of the Burgess bound. In a recent work, Khan \cite{RK} not only extended the range of $ P_2$, but also obtained the Weyl bound in the case of $ P_2 \sim P_3$.  By computing the second moment over a family of $ G L(2)$ forms,  Khan proved, in the range $ P_3 \gg P_2^{1/2}$, that
	\begin{align*}
		\sum_{f \in B_{k}^\star (P_2)} \left|  L\left( {1}/{2}, f \otimes \chi\right)\right|^2 \ll_{k, \epsilon}  \mathcal{Q}^\epsilon \left(P_2 + P_3 \right), 
	\end{align*} where $  B_{k}^\star (P_2)$ denote a basis of holomorphic newforms of level $P_2$ and weight $ k$, and $ \mathcal{Q} = P_2 P_3^2$. Recently, during an AIM workshop `` Delta Symbol and Subconvexity", the first and the third author used the  delta symbol approach to prove
	\begin{align*}
		L \left( {1}/{2}, f \otimes \chi \right) \ll_{ \epsilon} \mathcal{Q}^\epsilon  \frac{\sqrt{P_2 P_3} }{\min \left\lbrace \sqrt{P_2}, \sqrt{P_3} \right\rbrace}.  \ \ \ \ 
	\end{align*}
	This is of same strength as Khan \cite{RK}. 
	
%
\begin{ack} 
	This paper originated from discussions at the AIM online workshop `Delta Symbols and Subconvexity' held during 2-6 November 2020. The authors wish to thank the American Institute of Mathematics and the organizers of the workshop for their kind invitation. The authors also thank the participants of the workshop, especially Roman Holowinsky and Philippe Michel, for many enlightening conversations. For this work, S. K. Singh was partially  supported by D.S.T. inspire faculty fellowship no.   DST/INSPIRE/$04/2018/000945$ and R. Munshi was supported by J.C. Bose fellowship JCB/2021/000018 from SERB DST. Lastly the authors would like to thank the anonymous referee for a careful reading of the paper which helped in improving the exposition of the paper. 
\end{ack}

%
%
%
	\section{The Set-up}
Let $F$ and $f$ be as in Theorem \ref{main theorem}. We will denote the normalized Fourier coefficients of $f$ by $\lambda_f(n)$, and that of $F$ by $\lambda_F(n,r)$. The Rankin-Selberg convolution is given by the absolutely converging Dirichlet series 
$$
L \left( s, F \times f \right)=\mathop{\sum\sum}_{n,r=1}^\infty \frac{\lambda_F(n,r)\lambda_f(n)}{(nr^2)^s}
$$ 
in the right half plane $\text{Re}(s)=\sigma>1$. Here it is also given by a degree six Euler product. This function extends to an entire function and satisfies a functional equation of Riemann type. It is known that this Rankin-Selberg convolution is the standard $L$-function of a $GL(6)$ automorphic form \cite{KimS}.

\subsection{Approximate functional equation}
The functional equation gives an expression of the central value $L \left( {1}/{2}, F \times f \right)$ in terms of rapidly decaying series, the so called approximate functional equation (Theorem 5.3 from \cite{IK}). Taking a smooth dyadic subdivision of this expression we get the following. 
	\begin{lemma}\label{AFE}
		Let $ \mathcal{Q} = P_1^2 P_2^3 $ be the arithmetic  conductor attached to the $L$-function $L \left( {1}/{2}, F \times f \right)$. Then, as $\mathcal{Q} \rightarrow \infty$, we have
		\begin{align} \label{aproxi} 
			L \left( {1}/{2}, F \times f \right) \ll_{\epsilon} \mathcal{Q}^{\epsilon}\sum_{r \leq \mathcal{Q}^{(1+2\epsilon)/4}}\frac{1}{r}  \sup_{ N\leq \frac{\mathcal{Q}^{1/2+\epsilon}}{r^2}}\frac{|S_r(N)|}{N^{1/2}} +\mathcal{Q}^{-2021},
		\end{align}
		where $S_r(N)$ is a  sum of the form
		\begin{align} \label{s(n)-sum}
			S_r(N): = \mathop{\sum }_{n=1}^{\infty} \lambda_F(n,r) \lambda_f(n)  V \left(\frac{n}{N}\right),
		\end{align}
		for some smooth function $V$ supported in $[1,2]$ and satisfying $V^{(j)}(x) \ll_{j} 1$. 
	\end{lemma} 
	This is the usual starting point of the delta symbol approach. Thus, to get subconvexity, it is enough to get some cancellation in the sum
\begin{align*}
	S_r(N)=\sum_{n=1}^\infty \lambda_F(n,r)\lambda_f(n)V(n/N),
\end{align*}
for $N$ near the generic range $N\asymp \mathcal{Q}^{1/2}$.
	\subsection{Delta Symbol}
	Next we separate the oscillations involved in $S_r(N)$. For this we will use a Fourier expansion of the Kronecker delta symbol. For any $Q>1$ one has
\begin{align*} 
	\delta(n)= \frac{1}{Q} \sum_{1 \leq q \leq Q} \frac{1}{q} \, \sideset{}{^\star}{\sum}_{a \,{ \rm mod} \, q} \, e \left(\frac{an}{q}\right) \int_{\mathbb{R}} g(q,x) \,  e\left(\frac{nx}{qQ}\right) \, \mathrm{d}x, 
\end{align*}
where $g(q,x)$ is a smooth function of $x$ satisfying
\begin{align} \label{g properties}
	&g(q,x)=1+h(q,x), \quad \text{with} \ \ \  h(q,x)=O \left(\frac{Q}{q} \left(\frac{q}{Q}+|x|\right)^{B}\right), \notag \\
	& x^j \frac{\partial ^j}{\partial x^j}g(q,x) \ll \log Q \min \left\lbrace \frac{Q}{q}, \frac{1}{|x|}\right\rbrace, \notag  \\
	& g(q,x) \ll |x|^{-B},
\end{align}
for any $B>1$ and $j \geq 1$. (Here $e(z)=e^{2\pi iz}$).  This expansion of $\delta$ is due to Duke, Friedlander and Iwaniec, and one can find details of this in \cite{IK}. Using the third property of $g(q,x)$, we observe that the effective range of the integration over $x$   is $[-Q^{\varepsilon},Q^{\varepsilon}]$. Also it follows that if $q \ll Q^{1-\varepsilon}$ and $x \ll Q^{-\varepsilon}$, then $g(q,x)$ can be replaced by $1$ at the cost of a negligible error term. In the complimentary range, using second property, we have 
$$ x^j \frac{\partial ^j}{\partial x^j}g(q,x) \ll Q^{\varepsilon}.$$
Finally as in \cite{munshi12}, by Parseval and Cauchy, we get 
\begin{align*}
	\int_{\mathbb{R}}(|g(q,x)|+|g(q,x)|^2)\,  \mathrm{d}x \ll Q^{\varepsilon},
\end{align*}
i.e., $g(q,x)$ has average size `one' in the $L^1$ and $L^2$ sense. Applying this expansion and choosing $Q=N^{1/2}$, we get
\begin{align}\label{SN after delta}
	S_r(N)&=\mathop{\sum\sum}_{m,n=1}^\infty \lambda_F(n,r)\lambda_f(m)V(n/N)W(m/N) \delta(n-m) \notag \\
	&=\frac{1}{Q} \int_{\mathbb{R}}\sum_{1 \leq q \leq Q} \frac{g(q,x)}{q} \, \sideset{}{^\star}{\sum}_{a \, {\rm mod} \, q} \sum_{n=1}^{\infty}\lambda_F(n,r)e \left(\frac{n a}{q}\right)e\left(\frac{n x}{qQ}\right)V\left(\frac{n}{N}\right) \notag \\
	& \hspace{2cm}  \times  \sum_{m=1}^{\infty}\lambda_f(m)e \left(\frac{-ma}{q}\right)e\left(\frac{-m x}{q Q}\right)W\left(\frac{m}{N}\right)\mathrm{d}x.
\end{align}

	\subsection{Ideas behind the proof}
	In this section, we will discuss the method and present a sketch of the proof. For simplicity, let's consider the generic case,  i.e., $N=\sqrt{P_1^2P_2^3}$, $r=1$ and $q \asymp Q=\sqrt{N}$.  Thus $S_r(N) $ in \eqref{SN after delta} looks like 
	\begin{align*}
		\frac{1}{Q^2}\sum_{q \sim Q} \,  \sideset{}{^\star}{\sum}_{a \, {\rm mod} \, q} \, \sum_{n \sim N} \lambda_F(n,1)  e\left(\frac{an}{q}\right) \, \sum_{m \sim N} \, \lambda_f(m)  e\left(\frac{-am}{q}\right).
	\end{align*}
	On applying $GL(3)$ Voronoi to the $n$-sum, the  dual length becomes 
	$$n^\star \sim \frac{\mathrm{Conductor}}{\mathrm{Initial \  Length}}=\frac{Q^3P_1}{N}=P_1N^{1/2},$$ 
	and we save 
	$$\frac{\mathrm{Initial \ Length}}{\mathrm{\sqrt{Conductor}}}=\frac{N}{Q^{3/2}P_1^{1/2}}.$$
	Next we apply $GL(2)$ Voronoi formula to the sum over $m$. In this case, the  dual length (generic) is given by 
	$$ m^\star \sim \frac{Q^2P_2}{N}=P_2,$$
	and we save  $\frac{N}{Q\sqrt{P_2}}$ in this step. The resulting character sum is given by
	$$\sideset{}{^\star}{\sum}_{a \, {\rm mod} \, q} S(-\bar{a}\bar{P_1},n^\star;q)e\left(\frac{m^\star\bar{a}\bar{P_2}}{q}\right)=qe\left(\frac{\overline{P_1m^\star}P_2n^\star}{q}\right).$$
		This reduction of the character sum into an additive character with respect to the $GL(3)$ variable $n^\star$ drives the rest of the argument.  We save $\sqrt{Q}$ from the sum over $a$. Hence, in total, we have saved
	$$\frac{N}{Q^{3/2}P_1^{1/2}}\times \frac{N}{Q\sqrt{P_2}} \times \sqrt{Q}=\frac{N}{\sqrt{P_1P_2}}.$$
	In the next step, we apply Cauchy's inequality to the $n^\star$-sum  in the following resulting expression:
	\begin{align*}
		\sum_{q \sim Q}  \sum_{n^\star \sim P_1\sqrt{N}} \lambda_F(n,1) \, \sum_{m^\star \sim P_2} \, \lambda_f(m)e\left(\frac{\overline{P_1m^\star}P_2n^\star}{q}\right).
	\end{align*}
	After Cauchy, we arrive at 
	\begin{align*}
		(P_1\sqrt{N})^{1/2}\left(\sum_{n^\star \sim P_1\sqrt{N}}\Big| \sum_{q \sim Q}\sum_{m^\star \sim P_2} \, \lambda_f(m)e\left(\frac{\overline{P_1m^\star}P_2n^\star}{q}\right) \Big|^2\right)^{1/2},
	\end{align*}
	in which we seek to save $\sqrt{P_1P_2}$ and a little more. In the final step, we apply Poisson summation formula to the $n^\star$-sum. In the zero frequency($n^{\star}=0$), we save $(QP_2)^{1/2}$ which is sufficient provided
	$$(QP_2)^{1/2} >(P_1P_2)^{1/2} \iff  Q >P_1 \iff P_2^{3/2}>P_1.$$ 
	In the non-zero frequency, we save $\left({P_1\sqrt{N}}/{\sqrt{Q^2}}\right)^{1/2}$. From the additive character inside the modulus, which arises due to a specific feature of $GL(3) \times GL(2)$ $L$-functions, we also save $\sqrt{Q}$. Thus we save $(P_1\sqrt{N})^{1/2}$, which is sufficient if
	$$(P_1\sqrt{N})^{1/2} >(P_1P_2)^{1/2} \iff P_1 >P_2^{1/2}.$$
	Hence, we obtain subconvexity in the range $P_2^{1/2}<P_1<P_2^{3/2}$.
	Optimal saving, from Poisson, can be chosen  by taking the minimum of the zero and non-zero frequencies savings. Hence 
	$$S(N) \ll \frac{N\sqrt{P_1P_2}}{\min\left\lbrace \sqrt{QP_2},\, \sqrt{P_1\sqrt{N}}\right\rbrace}= \frac{N}{\min\left\lbrace N^{1/4}/P_1^{1/2},\, N^{1/4}/P_2^{1/2} \right\rbrace},$$
	and consequently
	$$L(1/2,F\times f) \ll \frac{(P_1^2P_2^3)^{1/4}}{\min\left\lbrace {P_2^{3/8}}/{P_1^{1/4}}, \, {P_1^{1/4}}/{P_2^{1/8}}\right\rbrace},$$
	which is best possible when $P_1 \asymp P_2(:=P)$ and $P_1 \neq P_2$. In this case we get 
	$$L(1/2,F\times f) \ll_\epsilon \left(P^5\right)^{1/4-1/40+\epsilon}.$$
	
	\section{ Voronoi Summation Formula}
Our next step involves applications of summation formulas. 
	
	\subsection{GL(3) Voronoi}
	In this section, we analyze the sum over $n$ using $GL(3)$ Voronoi summation formula.
	The following Lemma, except for the notations, is taken from \cite{FZ}. Let $F$ be a Hecke-Maass cusp form of type $(\nu_{1}, \nu_{2})$ for the congruent subgroup $\Gamma_0(P_1)$ of $SL(3,\mathbb{Z})$ with the trivial character. The Fourier coefficients of $F$ and that of its dual $\tilde{F}$ are related by
	$$\lambda_F(r,n)=\lambda_{\tilde{F}}(n,r),$$
	for $(nr,P_1)=1$. Let 
	$${\alpha}_{1} = - \nu_{1} - 2 \nu_{2}+1, \, {\alpha}_{2} = - \nu_{1}+ \nu_{2},  \, {\alpha}_{3} = 2 \nu_{1}+ \nu_{2}-1$$ 
	be the Langlands parameters for $F$ (see Goldfeld \cite{DG} for more details).
	Let $g$ be a compactly supported smooth function on  $ (0, \infty )$ and $\tilde{g}(s) = \int_{0}^{\infty} g(x) x^{s-1} \mathrm{d}x$ be its Mellin transform. For $\ell= 0$ and $1$, we define
	\begin{equation*}
		\gamma_{\ell}(s) := i^\ell \varepsilon(F)P_1^{1/2+s}\frac{\pi^{-3s-\frac{3}{2}}}{2} \, \prod_{i=1}^{3} \frac{\Gamma\left(\frac{1+s+{\alpha}_{i}+ \ell}{2}\right)}{\Gamma\left(\frac{-s-{\alpha}_{i}+ \ell}{2}\right)},
	\end{equation*}
	with $|\varepsilon(F)|=1$. Set $\gamma_{\pm}(s) = \gamma_{0}(s) \mp \gamma_{1}(s)$ and let 
	$$H_{\pm}(y) = \frac{1}{2 \pi i} \int_{(\sigma)} y^{-s} \, \frac{\pi^{-3s-\frac{3}{2}}}{2} \,\gamma_{\pm}(s) \tilde{g}(-s) \, \mathrm{d}s, $$
	where $\sigma > -1 + \max \{-\Re({\alpha}_{1}), -\Re({ \alpha}_{2}), -\Re({\alpha}_{3})\}$.  Let $G_{\pm}(y)=P_1^{1/2}H_{\pm}(y/P_1)$. With the aid of the above terminology, we now state the $GL(3)$ Voronoi summation formula in the following lemma:
	\begin{lemma} \label{gl3voronoi}
		Let $g(x)$ and  $\lambda_{F}(n,r)$ be as above. Let $a, q \in \mathbb{Z}$ with $q> 0, (a,q)=1,$  and let $\bar{a}$ be the multiplicative inverse of $a$ modulo $q$. Suppose $(qr,P_1)=1$. Then we have
		\begin{align*} \label{GL3-Voro}
			\sum_{n=1}^{\infty} \lambda_{F}(n,r) e\left(\frac{an}{q}\right) g(n) 
			=q  \sum_{\pm} \sum_{n_{1}|qr} \sum_{n_{2}=1}^{\infty}  \frac{\lambda_{{F}}(n_1,n_2)}{n_{1} n_{2}} S\left(r \bar{a}\bar{P_1}, \pm n_{2}; qr/n_{1}\right) G_{\pm} \left(\frac{n_{1}^2 n_{2}}{q^3 r}\right)
		\end{align*} 
		where  $S(a,b;q)$ is the  Kloosterman sum which is defined as: 
		$$S(a,b;q) = \sideset{}{^\star}{\sum}_{x \,{\rm mod} \, q} e\left(\frac{ax+b\bar{x}}{q}\right).$$
	\end{lemma}
	\begin{proof}
		See \cite{FZ} for the proof. 
	\end{proof}
	To apply  Lemma \ref{gl3voronoi} in our setup, we need to extract the  oscillations of the  integral transform. To this end, we state the following lemma. 
	\begin{lemma} \label{GL3oscilation}
		Let $g$ be supported in the interval $[X,2X]$ and let $H_\pm$ be defined as above. Then for any fixed integer $K \geq 1$ and $xX \gg 1$, we have
		\begin{equation*}
			H_{\pm}(x)=  x \int_{0}^{\infty} g(y) \sum_{j=1}^{K} \frac{c_{j}({\pm}) e\left(3 (xy)^{1/3} \right) + d_{j}({\pm}) e\left(-3 (xy)^{1/3} \right)}{\left( xy\right)^{j/3}} \, \mathrm{d} y + O \left((xX)^{\frac{-K+2}{3}}\right),
		\end{equation*}
		where $c_{j}(\pm)$ and $d_{j}(\pm)$ are some absolute constants depending on $\alpha_{i}$, for $i=1, 2, 3$.  
	\end{lemma}
	\begin{proof}
		See  Lemma 6.1 of \cite{XL}.
	\end{proof}
	Plugging the leading term of  Lemma \ref{GL3oscilation} in Lemma \ref{gl3voronoi} and using the resulting expression in \eqref{SN after delta} we see that the sum over $n$ gets transformed into
	\begin{align}\label{dual n sum}
		\frac{N^{2/3}}{P_1^{1/6}qr^{2/3}} \sum_{\pm} \sum_{n_{1}|qr}n_1^{1/3} \sum_{n_{2}=1}^{\infty}  \frac{\lambda_F(n_1,n_2)}{ n_{2}^{1/3}} S\left(r \bar{a}\bar{P_1}, \pm n_{2}; qr/n_{1}\right)\mathcal{I}(...),
	\end{align}
	where 
	\begin{align*}
		\mathcal{I}(...)=\int_0^\infty V(z)e\left(\frac{Nxz}{qQ}\pm \frac{3(Nn_1^2n_2z)^{1/3}}{P_1^{1/3}qr^{1/3}}\right)\mathrm{d}z.
	\end{align*}
We observe that, using integration by parts repeatedly, the above integral is negligibly small if 
	$$n_1^2 n_2 \gg N^\epsilon{\sqrt{N}P_1r}=N^\epsilon\frac{P_1Q^3r}{N}=:N_0.$$
	
		In the case when $P_1|qr$, an appropriate Voronoi summation from \cite{FZ} can still be used. In fact it turns out that our analysis in this paper still goes through with slight modification and the final bound is even better.  As such we proceed to present our analysis only in the coprime case.
	\subsection{GL(2) Voronoi} \label{sec-gl2-vor}
	In this section, we dualize the sum over $m$ using $GL(2)$ Voronoi summation formula. 
	
	\begin{lemma} \label{voronoi hol}
		Let $f\in H_k (P_2)$  be a holomorphic Hecke cuspform with Fourier coefficients $\lambda_f(n)$ and trivial nebentypus. Let $a$ and $q$ be integers with $(aP_2, q)=1$. Let $g$ be a compactly supported smooth bump function on $\mathbb{R}$. Then  we have 
			\begin{equation} 
			\sum_{m=1}^\infty \lambda_f (m) e\left( \frac{-am}{q}\right) g(n) = \frac{1}{q} \frac{\eta_f(P_2)}{\sqrt{P_2}} \sum_{n=1}^\infty \lambda_{f}( m) e\left(  \frac{m\overline{a P_2}}{q}\right) H\left( \frac{m}{P_2q^2}\right),
		\end{equation}
		where $ a \overline{a} \equiv 1 (\textrm{mod} \  q)$, $|\eta_f(P_2)|=1$ and 
		\begin{align*} 
			H (y)=   2 \pi i^k \int_0^\infty g(x)  J_{k-1} \left( 4\pi \sqrt{xy}\right) \mathrm{d}x,
		\end{align*} where $ J_{k-1} $  is the $J$-Bessel function and $k$ is the weight of $f$.   
	\end{lemma} 
	\begin{proof}
		See appendix of \cite{KMV}. 
	\end{proof}
 Extracting the oscillatons of $J_{k-1}$, 
 $$J_{k-1}(2\pi x)=e(x)W_{k-1}(x)+e(-x)\bar{W}_{k-1}(x),$$
 with 
 $$x^j\frac{\mathrm{d}^j}{\mathrm{d}x^j}W_{k-1}(x) \ll_{j, k} 1/\sqrt{x},$$
 we see that $H(y)$ can be essentially replaced by 
\begin{align*}
	H(y)=\frac{2 \pi i^k }{{y}^{1/4}}\int_0^\infty g_1(x)  e\left(\pm 2 \sqrt{xy}\right)\mathrm{d}x,
\end{align*}
in our analysis,  where  $g_1$ is the new weight  function which has compact support and $x^jg_1^{(j)} (x) \ll_j 1$, $j \geq 0$. Applying the above lemma, the sum over $m$ in \eqref{SN after delta} reduces to 
	\begin{align}\label{dual m sum}
		\frac{N^{3/4}{\eta_f(P_2)}}{P_2^{1/4}\sqrt{q}} \sum_{m=1}^\infty \frac{\lambda_{f}( m)}{{m}^{1/4}} e\left(  \frac{m\overline{a P_2}}{q}\right)\int_0^\infty W(y)e\left(\frac{-yNx}{qQ}\right)  e\left(\frac{\pm 2 \sqrt{Nmy}}{qP_2^{1/2}}\right) \mathrm{d}y.
	\end{align}
	Notice the  abuse of notation: the weight function W  is different from the one in  \eqref{SN after delta}. Using  stationary phase analysis we observe that the above integral is negligibly small unless 
	$$m \ll N^{\epsilon}P_2=N^{\epsilon}\frac{Q^2P_2}{N}=:M_0.$$

			Again we will ignore the degenerate case where $P_2|q$ and proceed with the analysis of the generic case. Indeed our analysis works in the degenerate case as well, and the bound that we obtain is even better (as one will expect).

	Now plugging \eqref{dual n sum} and \eqref{dual m sum} in \eqref{SN after delta}, we arrive at 
	\begin{align}\label{SN after voronoi}
		&\frac{N^{17/12}\eta_f(P_2)}{P_1^{1/6}P_2^{1/4}Qr^{2/3}} \sum_{1 \leq q \leq Q} \frac{1}{q^{5/2}} \sum_{\pm} \sum_{n_{1}|qr}n_1^{1/3} \sum_{n_{2} \ll N_0/n_1^2}  \frac{\lambda_{{F}}(n_1,n_2)}{ n_{2}^{1/3}} \sum_{m \ll M_0} \frac{\lambda_{f}( m)}{{m}^{1/4}}\mathcal{C}(...)\mathfrak{I}(...),
	\end{align}
	where the integral transform is given by
	\begin{align*}
		\mathfrak{I}(...)=&\int_{\mathbb{R}}W(x)g(q,x) \int_0^{\infty}W(y) \\
		& \times \int_0^{\infty}V(z)
		e\left(\frac{Nx(z-y)}{qQ}\pm \frac{ 2 \sqrt{Nmy}}{qP_2^{1/2}}\pm \frac{3(Nn_1^2n_2z)^{1/3}}{P_1^{1/3}qr^{1/3}}\right)\mathrm{d}z\,\mathrm{d}y\,  \mathrm{d}x,
	\end{align*}
	and the character sum is given by
	\begin{align*}
		\mathcal{C}(...):&=\sideset{}{^\star}{\sum}_{a \, {\rm mod} \, q}S\left(r \bar{a}\bar{P_1}, \pm n_{2}; qr/n_{1}\right)e\left(  \frac{m\bar{a}\overline{ P_2}}{q}\right) \\
		&=\sum_{d|q}\, d\mu\left(\frac{q}{d}\right)\sideset{}{^ \star}\sum_{\substack{\alpha \, {\rm mod} \, qr/n_1 \\ \bar{P_1}n_1\alpha\equiv-m\bar{P_2} \, {\rm mod} \, d}}e\left(\pm\frac{\bar{\alpha}n_2}{qr/n_1}\right).
	\end{align*}
	\section{Cauchy and Poisson}
	\subsection{Cauchy inequality}
	Now we apply Cauchy's inequality to the $n_2$-sum in \eqref{SN after voronoi}. To this end, we split the sum over $q$ into  dyadic blocks $q\backsim C$ and further writing $q=q_1q_2$ with $q_1|(n_1r)^\infty$, $(n_1r,q_2)=1$, we see that $S_r(N)$   is bounded by 
	\begin{align}\label{Q-dyadic block}
		\sup_{C \ll Q}\frac{N^{17/12}}{P_2^{1/4}P_1^{1/6}Qr^{2/3}C^{5/2}}\sum_{\pm}\sum_{\frac{n_1}{(n_1,r)}\ll C}n_1^{1/3}\sum_{\frac{n_1}{(n_1,r)}|q_1|(n_1r)^\infty}\sum_{n_2\ll {N_0}/{n_1^2}} \frac{|\lambda_F(n_1,n_2)|}{n_2^{1/3}} \notag \\ 
		\times \Big|\sum_{q_2\backsim C/q_1}\sum_{m \ll M_0}\frac{\lambda_{f}(m)}{m^{1/4}}\mathcal{C}(...)\mathfrak{I}(...)\Big|,
	\end{align}
	On applying the Cauchy's inequality to the $n_2$-sum we arrive at
	
	\begin{align}\label{S(N) after cauchy}
		S_r(N)\ll 	\sup_{C \ll Q}\frac{N^{17/12}}{P_2^{1/4}P_1^{1/6}Qr^{2/3}C^{5/2}}\sum_{\pm}\sum_{\frac{n_1}{(n_1,r)}\ll C}n_1^{1/3}\Theta^{1/2}\sum_{\frac{n_1}{(n_1,r)}|q_1|(n_1r)^\infty}\sqrt{\Omega},
	\end{align}
	where 
	\begin{align}\label{theta}
		\Theta=\sum_{n_2\ll N_0/n_1^2} \frac{|\lambda_F(n_1,n_2)|^2}{n_2^{2/3}},
	\end{align} 
	and 
	\begin{align}\label{omega}
		\Omega=\sum_{n_2\ll N_0/n_1^2}\Big|\sum_{q_2\backsim C/q_1}\sum_{m \ll M_0}\frac{\lambda_{f}(m)}{m^{1/4}}\mathcal{C}(...)\mathfrak{I}(...)\Big|^2.
	\end{align}
	\subsection{Poisson}
	We now  apply the Poisson summation formula to the $n_2$-sum  in $\eqref{omega}$. To this end, we smooth out the $n_2$-sum, i.e., we plug in an appropriate smooth bump function, say, $W$. Opening the absolute value square, we get 
	\begin{align*}
		\Omega=&\mathop{\sum \sum}_{q_2,q_2^{\prime}\backsim C/q_1}\mathop{\sum \sum}_{m,m^{\prime} \ll M_0 }\frac{\lambda_f(m)\lambda_f(m^\prime)}{(mm^\prime)^{1/4}} 
		\\
		& \times \sum_{n_2 \in \mathbb{Z}}W\left(\frac{n_2}{N_0/n_1^2}\right)\mathcal{C}(...)\overline{\mathcal{C}}(...)\mathfrak{I}(...)\overline{\mathfrak{I}}(...),
	\end{align*}
	Reducing $n_2$ modulo $q_1q_2q_2^\prime r/n_1:=\gamma$, and using the change of variable $$n_2 \mapsto n_2q_1q_2q_2^\prime r/n_1+\beta, \ \  \mathrm{with}, \ \  0 \leq \beta < q_1q_2q_2^\prime r/n_1,$$ followed by the Poisson summation formula, we arrive at
	\begin{align}\label{final omega}
		\Omega_{}=&\mathop{\sum \sum}_{q_2,q_2^{\prime}\backsim C/q_1}\mathop{\sum \sum}_{m,m^{\prime} \ll M_0}\frac{\lambda_f(m)\lambda_f(m^\prime)}{(mm^\prime)^{1/4}} \sum_{n_2 \in \mathbb{Z}}\sum_{\beta \, \mathrm{mod}\,\gamma}\mathcal{C}(...)\overline{\mathcal{C}}(...)\mathcal{J},
	\end{align}
	where 
	\begin{align*}
		\mathcal{J}&=\int_{\mathbb{R}}W\left(\frac{w\gamma+\beta}{N_0/n_1^2}\right)\mathfrak{I}(...)\overline{\mathfrak{I}}(...)e(-n_2w) \, \mathrm{d}w.
	\end{align*}
	Now changing the variable 
	$$\frac{w\gamma+\beta}{N_0/n_1^2} \mapsto w,$$
	we arive at 
	\begin{align*}
		\mathcal{J}&=\frac{N_0}{n_1^2 \gamma}e\left(\frac{n_2\beta}{\gamma}\right)\int_{\mathbb{R}}W\left({w}\right) \mathfrak{I}(...)\overline{\mathfrak{I}}(...) e\left(\frac{-n_2N_0w}{n_1^2\gamma}\right)\mathrm{d}w.
	\end{align*}
	Plugging this back in \eqref{final omega}, and executing the sum over $\beta$, we arrive at 
	\begin{align}\label{final simplified omega}
		\Omega_{}=\frac{N_0}{n_1^2}\mathop{\sum \sum}_{q_2,q_2^{\prime}\backsim C/q_1}\mathop{\sum \sum}_{m,m^{\prime} \ll M_0}\frac{\lambda_f(m)\lambda_f(m^\prime)}{(mm^\prime)^{1/4}} \sum_{n_2 \in \mathbb{Z}}\mathfrak{C} \, \mathcal{I},
	\end{align}
	where
	\begin{align}\label{character sum}
		\mathfrak{C}=\mathop{\sum \sum}_{\substack{d|q \\ d^{\prime}|q^\prime}}dd^\prime\mu\left(\frac{q}{d}\right)\mu\left(\frac{q^\prime}{d^\prime}\right)\mathop{\sideset{}{^ \star}\sum_{\substack{\alpha \, {\rm mod} \, qr/n_1 \\ \bar{P_1} n_1\alpha\equiv-m\bar{P_2} \, {\rm mod} \, d}} \  \sideset{}{^ \star}\sum_{\substack{\alpha^\prime \, {\rm mod} \, q^\prime r/n_1 \\ \bar{P_1} n_1\alpha^\prime \equiv-\bar{P_2} m^\prime \, {\rm mod} \, d^\prime}}}_{\pm \bar{\alpha}q_2^\prime \mp\bar{\alpha}^\prime q_2\equiv -n_2  \, {\rm mod} \, q_1q_2q_2^\prime r/n_1}1
	\end{align}
	and 
	\begin{align}\label{final integral}
		\mathcal{I}=\int_{\mathbb{R}}W(w) \mathfrak{I}(...)\overline{\mathfrak{I}}(...)\,    e\left(\frac{-n_2N_0w}{n_1q_1q_2q_2^\prime r}\right)\mathrm{d}w.
	\end{align}
	On applying integration by parts, we see that the above integral is negligibly small if 
	\begin{align}\label{N_2}
		n_2 \gg \frac{Q}{q}\frac{n_1q_1q_2q_2^\prime r}{N_0}:=N_2.
	\end{align}
	\section{Bounding the integral}
	In this section we will analyze the integral $	\mathcal{I}$ given in \eqref{final integral}.
	Recall that the  integral $\mathfrak{I}(...)$ is given by
	\begin{align}\label{int before cauchy}
		\mathfrak{I}(...)=&\int_{\mathbb{R}}W(x)g(q,x) \int_0^{\infty}W(y) \notag \\
		& \times \int_0^{\infty}V(z)
		e\left(\frac{Nx(z-y)}{qQ}\pm \frac{ 2 \sqrt{Nmy}}{qP_2^{1/2}}\pm \frac{3(NN_0wz)^{1/3}}{P_1^{1/3}qr^{1/3}}\right)\mathrm{d}z\, \mathrm{d}y\, \mathrm{d}x. 
	\end{align}
	Let's first focus on $x$-integral, i.e.,
	$$\int_{\mathbb{R}}W(x)g(q,x)e\left(\frac{Nx(z-y)}{qQ}\right)\mathrm{d}x.$$
	In the case, $q \ll Q^{1-\epsilon}$, we split the above integral as follows:
	\begin{align*}
		\left(\int_{|x| \ll Q^{-\epsilon}}+ \int_{|x| \gg Q^{-\epsilon}}\right) W(x)g(q,x)e\left(\frac{Nx(z-y)}{qQ}\right)\mathrm{d}x.
	\end{align*}
	For the first part, we can replace $g(q,x)$ by $1$ at the cost of a negligible error term (see \eqref{g properties}) so that we essentially have 
	\begin{align*}
		\int_{ |x| \ll Q^{-\epsilon}} W(x)e\left(\frac{Nx(z-y)}{qQ}\right)\mathrm{d}x.
	\end{align*}
	Using integration by parts, we observe that the above integral is negligibly small unless 
	$$|z-y| \ll \frac{q}{Q}Q^\epsilon.$$
	For the second part, using $g^{(j)}(q,x) \ll Q^{\epsilon j}$, we get the restriction $|z-y| \ll \frac{q}{Q}Q^\epsilon$. In the other case, i.e., $q \gg Q^{1-\epsilon}$, the condition $|z-y| \ll \frac{q}{Q}Q^\epsilon$ is trivially true. Now we write $z$ as $z=y+u$, with $|u| \ll  \frac{q}{Q}Q^\epsilon $. Thus the integral  $\mathfrak{I}(...)$ up to a negligible error term is given by  
	\begin{align}\label{simplified integral }
		\int_{\mathbb{R}}W(x)g(q,x) \int_0^{\infty} \int_{|u| \ll qQ^{\epsilon}/Q}V(y+u)W(y)
		e\left(\frac{Nxu}{qQ}\right) \notag\\
		\times e\left(\pm \frac{ 2 \sqrt{Nmy}}{qP_2^{1/2}}\pm \frac{3(NN_0w(y+u))^{1/3}}{P_1^{1/3}qr^{1/3}}\right)\mathrm{d}u \, \mathrm{d}y\, \mathrm{d}x. 
	\end{align}
	Now we consider the $y$-integral
	$$\int_{\mathbb{R}}V(y+u)W(y)e\left(\pm \frac{ 2 \sqrt{Nmy}}{qP_2^{1/2}}\pm \frac{3(NN_0w(y+u))^{1/3}}{P_1^{1/3}qr^{1/3}}\right)\mathrm{d}y.$$
	Expanding $(y+u)^{1/3}$ into the  Taylor series 
	$$(y+u)^{1/3}=y^{1/3}+\frac{u}{3y^{2/3}}-\frac{u^2}{9y^{5/3}}+\dots,$$
	we observe that it is enough   to consider only the leading term as 
	\begin{align*}
		\frac{3(NN_0)^{1/3}}{P_1^{1/3}qr^{1/3}}\frac{u}{3y^{2/3}}\ll \frac{Qu}{q}\ll Q^\varepsilon.
	\end{align*}
Thus  we  are required to  analyze  the integral
	\begin{align}\label{int y}
		I=\int_{\mathbb{R}}W(y)e\left(\pm \frac{ 2 \sqrt{Nmy}}{qP_2^{1/2}}\pm \frac{3(NN_0wy)^{1/3}}{P_1^{1/3}qr^{1/3}}\right)\mathrm{d}y. 
	\end{align}
	By stationary phase analysis we see that the integral is negligibly small unless  
$$\frac{ 2 \sqrt{Nm}}{qP_2^{1/2}} \asymp  \frac{3(NN_0)^{1/3}}{P_1^{1/3}qr^{1/3}}\approx \frac{Q}{q}.$$ 
Thus the above integral is negligibly small unless $m \sim M_0$ (with $M_0$ as in Section~\ref{sec-gl2-vor}), in which case the above $y$-integral is bounded by  
	$$I \ll \frac{\sqrt{q}}{\sqrt{Q}}.$$ 
	Hence, executing the remaining integrals trivially, and using
	$$\int_{\mathbb{R}}|g(q,x)|\mathrm{d}x \ll Q^{\epsilon},$$ we see that $\mathfrak{I}$ is bounded by 
	$$\mathfrak{I}(...)\ll q^{3/2}  / Q^{3/2}.$$
	On substituting this bound in \eqref{final integral}, we get 
	\begin{align} 
		\mathcal{I}\ll q^3/Q^3.
	\end{align}
	We record the above discussion in the following lemma.
	\begin{lemma}\label{integral bound for zero}
		Let $I$, $\mathfrak{I}(...)$ and $\mathcal{I}$ be as in \eqref{int y}, \eqref{int before cauchy} and \eqref{final integral} respectively. Then we have
		$$I \ll \frac{\sqrt{q}}{\sqrt{Q}},$$
		$$\mathfrak{I}(...)\ll q^{3/2}  / Q^{3/2},$$
		and 
		$$\mathcal{I}\ll q^{3}/Q^{3}.$$
	\end{lemma}
	\section{Character Sums}
	In this section, we will estimate  the character sum $\mathfrak{C}$ given in \eqref{character sum},
	\begin{align}
		\mathfrak{C}=\mathop{\sum \sum}_{\substack{d|q \\ d^{\prime}|q^\prime}}dd^\prime\mu\left(\frac{q}{d}\right)\mu\left(\frac{q^\prime}{d^\prime}\right)\mathop{\sideset{}{^ \star}\sum_{\substack{\alpha \, {\rm mod} \, qr/n_1 \\ \bar{P_1} n_1\alpha\equiv-m\bar{P_2} \, {\rm mod} \, d}} \  \sideset{}{^ \star}\sum_{\substack{\alpha^\prime \, {\rm mod} \, q^\prime r/n_1 \\ \bar{P_1} n_1\alpha^\prime \equiv-\bar{P_2} m^\prime \, {\rm mod} \, d^\prime}}}_{\pm \bar{\alpha}q_2^\prime \mp\bar{\alpha}^\prime q_2\equiv -n_2  \, {\rm mod} \, q_1q_2q_2^\prime r/n_1}1.
	\end{align}
	In the case, $n_2=0$, the congruence condition  $$\pm \bar{\alpha}q_2^\prime \mp\bar{\alpha}^\prime q_2\equiv 0  \, {\rm mod} \, q_1q_2q_2^\prime r/n_1$$ implies that $q_2=q_2^{\prime}$ and $\alpha=\alpha^{\prime}$. So we can bound the character sum  $\mathfrak{C}$   as 
	\begin{align} \label{character bound for zero}
		\mathfrak{C}  \ll   \mathop{\sum \sum}_{\substack{d ,d^{\prime} \vert q}} dd'  \mathop{\sideset{}{^\star} \sum_{\alpha \; {\rm mod} \; qr/n_1 } }_{\substack{\bar{P_1}n_1\alpha\equiv-m\bar{P_2} \, {\rm mod} \, d \\ \bar{P_1}n_1\alpha \equiv-\bar{P_2}m^\prime \, {\rm mod} \, d^\prime  }} 1 \ll \mathop{\sum \sum}_{\substack{d ,d^{\prime} \vert q \\ (d,d^{\prime})|(m-m^{\prime})}} dd' \frac{qr}{[d,d^{\prime}]}.
	\end{align}
	For $n_2 \neq 0$, we have the following lemma.
	\begin{lemma}\label{character bound for nonzero}
		Let $\mathfrak{C}$ be as in \eqref{character sum}. Then, for $n_2 \neq 0$,	we have 
		$$\mathfrak{C} \ll \frac{q_{1}^2 \, r (m,n_{1})}{n_{1}} \mathop{\sum \sum}_{\substack{d_{2} \mid (q_{2},  n_{1} q_{2}^{\prime}\mp mn_{2}P_1\bar{P_2}) \\ d_{2}^{\prime} \mid (q_{2}^{\prime},  n_{1} q_{2} \pm m^{\prime} n_{2}P_1\bar{P_2})}} \, d_{2} d_{2}^{\prime}.$$
	\end{lemma}
	\begin{proof}
		  Let's recall  from \eqref{character sum} that
		\begin{align*}
			\mathfrak{C}=\mathop{\sum \sum}_{\substack{d|q \\ d^{\prime}|q^\prime}}dd^\prime\mu\left(\frac{q}{d}\right)\mu\left(\frac{q^\prime}{d^\prime}\right)\mathop{\sideset{}{^ \star}\sum_{\substack{\alpha \, {\rm mod} \, qr/n_1 \\ \bar{P_1} n_1\alpha\equiv-m\bar{P_2} \, {\rm mod} \, d}} \  \sideset{}{^ \star}\sum_{\substack{\alpha^\prime \, {\rm mod} \, q^\prime r/n_1 \\ \bar{P_1} n_1\alpha^\prime \equiv-\bar{P_2} m^\prime \, {\rm mod} \, d^\prime}}}_{\pm \bar{\alpha}q_2^\prime \mp\bar{\alpha}^\prime q_2\equiv -n_2  \, {\rm mod} \, q_1q_2q_2^\prime r/n_1}1
		\end{align*}
		Using the Chinese Remainder theorem,  we observe that $\mathfrak{C}$ can be dominated by a product of two sums $ \mathfrak{C} \ll \mathfrak{C}^{(1)} \mathfrak{C}^{(2)}$,
		where
		$$\mathfrak{C}^{(1)} = \mathop{\sum \sum}_{\substack{d_{1}, d_{1}^{\prime} | q_{1}}} d_{1} d_{1}^{\prime}  \; \mathop{\sideset{}{^\star}{\sum}_{\substack{\beta \; \mathrm{ mod} \; \frac{q_{1}r}{n_{1}} \\  n_{1} \beta \;  \equiv \; - mP_1\bar{P_2} \; \mathrm{mod} \;  d_{1}}} \  \sideset{}{^\star}{\sum}_{\substack{\beta^{\prime} \; \mathrm{mod} \; \frac{q_{1}r}{n_{1}} \\  n_{1}   \beta^{\prime} \;  \equiv \; - m^{\prime}P_1\bar{P_2} \; {\rm mod} \;  d_{1}^{\prime}}}}_{\pm \overline{\beta} q_{2}^{\prime} \mp \overline{\beta^{\prime}} q_{2} + n_{2} \; \equiv \; 0 \;  {\rm mod} \;  {q_{1}r }/{n_{1}} } \; 1 $$
		and
		$$\mathfrak{C}^{(2)} = \mathop{\sum \sum}_{\substack{d_{2} \mid q_{2} \\ d_{2}^{\prime} \mid q_{2}^{\prime}}} d_{2} d_{2}^{\prime}  \; \mathop{\sideset{}{^\star}{\sum}_{\substack{\beta \; {\rm mod} \; q_{2} \\  n_{1} \beta \;  \equiv \; - m P_1\bar{P_2}\; {\rm mod} \;  d_{2}}} \, \sideset{}{^\star}{\sum}_{\substack{\beta^{\prime} \; {\rm mod} \; q_{2}^{\prime} \\  n_{1}   \beta^{\prime} \;  \equiv \; - m^{\prime}P_1\bar{P_2}  \; {\rm mod} \;  d_{2}^{\prime}}}}_{\pm \overline{\beta} q_{2}^{\prime} \mp \overline{\beta^{\prime}} q_{2} + n_{2} \; \equiv \; 0 \;  {\rm mod} \;  q_{2} q_{2}^{\prime} } \; 1.$$
		
		In the second sum $\mathfrak{C}^{(2)}$, since  $(n_{1},q_{2}q_{2}^{\prime})=1$,  we get  $\beta \equiv -m\bar{n_1}P_1\bar{P_2} \ \mathrm{mod} \, d_{2}$ and $\beta^{\prime} \equiv \, -m^{\prime}\bar{n_1}P_1\bar{P_2} \, \mathrm{mod}\, d_{2}^{\prime}$. Now using the congruence modulo $q_{2} q_{2}^{\prime}$, we conclude that
		
		$$\mathfrak{C}^{(2)} \ll  \mathop{\sum \sum}_{\substack{d_{2} \mid (q_{2},  n_{1} q_{2}^{\prime}\mp mn_{2}P_1\bar{P_2}) \\ d_{2}^{\prime} \mid (q_{2}^{\prime},  n_{1} q_{2} \pm m^{\prime} n_{2}P_1\bar{P_2})}} \, d_{2} d_{2}^{\prime}.$$
		In the first sum $\mathfrak{C}^{(1)}$, the congruence condition determines $\beta^{\prime}$ uniquely in terms of $\beta$, and hence 
		$$\mathfrak{C}^{(1)} \ll \mathop{\sum \sum}_{\substack{d_{1}, d_{1}^{\prime} | q_{1}}} d_{1} d_{1}^{\prime} \sideset{}{^\star}{\sum}_{\substack{\beta \; {\rm mod} \; {q_{1}r}/{n_{1}} \\  n_{1} \beta \;  \equiv \; - mP_1\bar{P_2} \; {\rm mod} \;  d_{1}}} \, 1 \ll \frac{r \, q_{1}^2 \, (m,n_{1}) }{n_{1}} .$$
		Hence we have the lemma.
	\end{proof}
	
	\section{Zero frequency}
	In this section we will estimate the contribution of the zero frequency $n_2= 0$ to $\Omega$ in \eqref{final simplified omega},  and thus estimate its total contribution to $S_r(N)$. We have the following lemma.
	\begin{lemma}\label{SN for zero frequency}
		Let $S_r(N)$ be as in \eqref{S(N) after cauchy}. The total contribution of the zero frequency $n_2=0$ to $S_r(N)$ is dominated by $O(r^{1/2}N^{3/4}\sqrt{P_1}).$
	\end{lemma}
	\begin{proof}
		On substituting bounds for $\mathcal{I}$ and $\mathfrak{C}$ from Lemma \ref{integral bound for zero} and  \eqref{character bound for zero} respectively into \eqref{final simplified omega},  we see that the contribution of $n_2=0$ to $\Omega$, is bounded by
		\begin{align*}
			&\ll \frac{N_0C^3}{n_1^2M_0^{1/2}Q^3}\mathop{\sum }_{\substack{q_2 \sim C/q_1 }} qr \mathop{\sum \sum}_{\substack{d ,d^{\prime} \vert q}}(d,d^{\prime})\mathop{\sum \sum}_{\substack{m,m^\prime \sim M_{0} \\ (d,d^{\prime})|(m-m^{\prime}) }}1 \notag \\
			&\ll	\frac{N_0C^3}{n_1^2M_0^{1/2}Q^3}\mathop{\sum }_{\substack{q_2 \sim C/q_1 }} qr  \mathop{\sum \sum}_{\substack{d ,d^{\prime} \vert q}}(M_0(d,d^{\prime})+M_0^2) \\
			&\ll 	\frac{N_0C^5rM_0^{1/2}}{n_1^2Q^3q_1}(C+M_0).
		\end{align*}
		Upon substituting this bound for $\Omega$ in \eqref{S(N) after cauchy}, we get 
		\begin{align*}
			&\sup_{C \ll Q}\frac{N^{17/12}}{P_2^{1/4}P_1^{1/6}Qr^{2/3}C^{5/2}}\sum_{\pm}\sum_{\frac{n_1}{(n_1,r)}\ll C}n_1^{1/3}\Theta^{1/2}\sum_{\frac{n_1}{(n_1,r)}|q_1|(n_1r)^\infty}\left(\frac{N_0C^5rM_0^{1/2}}{n_1^2Q^3q_1}(C+M_0)\right)^{1/2} \notag \\
			& \ll \frac{N^{17/12}}{P_2^{1/4}P_1^{1/6}Qr^{2/3}}\frac{N_0^{1/2}r^{1/2}M_0^{1/4}}{Q^{3/2}}\sum_{\frac{n_1}{(n_1,r)}\ll C}\frac{\Theta^{1/2}}{n_1^{2/3}}\sum_{\frac{n_1}{(n_1,r)}|q_1|(n_1r)^\infty}\frac{1}{\sqrt{q_1}}\left(\sqrt{Q}+\sqrt{M_0}\right) \\
			& \ll \frac{N^{17/12}}{P_2^{1/4}P_1^{1/6}Qr^{2/3}}\frac{N_0^{1/2}r^{1/2}M_0^{1/4}}{Q^{3/2}}\sqrt{Q}\sum_{\frac{n_1}{(n_1,r)}\ll C}\frac{\Theta^{1/2}}{n_1^{7/6}}\sqrt{(n_1,r)}.
		\end{align*}
			Note that (as in \cite{munshi12}) we have
		\begin{align}\label{theta bound}
			\sum_{n_1\ll Cr}\frac{(n_1,r)^{1/2}}{n_1^{7/6}}\Theta^{1/2} \ll \left[\sum_{n_1 \ll Cr}\frac{(n_1,r)}{n_1}\right]^{1/2}\left[\mathop{\sum \sum}_{n_{1}^{2} n_{2} \leq N_0} \frac{\vert \lambda_{F}(n_{1},n_{2})\vert ^{2} }{(n_1^2n_2)^{2/3}}\right]^{1/2}\ll \,N_0^{1/6}.
		\end{align}
		Using this bound, we see that the contribution of $n_2=0$ to $S_r(N)$ is bounded by
		\begin{align*}
			S_r(N)\ll	\frac{N^{17/12}}{P_2^{1/4}P_1^{1/6}Qr^{2/3}}\frac{N_0^{1/2}r^{1/2}M_0^{1/4}}{Q^{3/2}}\sqrt{Q}N_0^{1/6} \ll r^{1/2}N^{3/4}\sqrt{P_1}.
		\end{align*}
	\end{proof}
	\section{Non-Zero Frequencies}
	In this section we will estimate the contribution of the non-zero frequencies $n_2\neq 0$ to $\Omega$ in \eqref{final simplified omega}. We have the following lemma.

	\begin{lemma}\label{SN for non-zero frequency}
		Let $S_r(N)$ be as in \eqref{S(N) after cauchy}. The total contribution of $n_2\neq 0$, to $S_r(N)$ is dominated by $O(\sqrt{r}N^{3/4}\sqrt{P_2}).$
	\end{lemma}
	\begin{proof}
		On plugging in the bounds for the character sums and the  integrals from Lemma \ref{character bound for nonzero} and Lemma \ref{integral bound for zero} respectively into \eqref{final simplified omega},  
	we see that the contribution of $n_2\neq 0$ to $\Omega$ (which we denote by $\Omega_{\neq 0}$) is bounded by
		\begin{align*}
		 &  \frac{q_{1}^2 N_0rC^3} {n_{1}^3M_0^{1/2}Q^3}  \mathop{\sum \sum }_{\substack{q_2, \,  q_{2}^{\prime} \sim \frac{C}{q_{1}}  } }  \mathop{\sum \sum}_{\substack{d_{2} \mid q_{2} \\ d_{2}^{\prime} \mid q_{2}^{\prime}}} d_{2} d_{2}^{\prime} \,  \mathop{ \mathop{\sum \ \sum \ \  \ \sum}_{m,m^{\prime} \sim M_{0} \ 0\neq n_2 \ll N_2}}_{\substack{ n_{1} q_{2}^{\prime} \mp m n_{2}P_1\bar{P_2} \, \equiv \,  0 \, {\rm mod} \, d_{2} \\  n_{1} q_{2}\pm m^{\prime} n_{2}P_1\bar{P_2} \, \equiv \,  0 \, {\rm mod} \, d_{2}^{\prime}}} (m,n_{1}) .
		\end{align*}
		Further writing $q_2d_2$ in place of $q_2$ and $q_2^{\prime}d_2^{\prime}$ in place of $q_2^{\prime}$, we arrive at
		\begin{align}\label{omega nonzero}
			\Omega_{\neq 0} & \ll \frac{q_{1}^2 N_0rC^3} {n_{1}^3M_0^{1/2}Q^3} \mathop{\sum \sum}_{d_{2}, d_{2}^{\prime} \ll C/q_{1} } d_{2} d_{2}^{\prime} \, \mathop{\sum \sum }_{\substack{q_{2} \sim \frac{C}{d_{2}q_{1}} \\ q_{2}^{\prime} \sim \frac{C}{d_{2}^{\prime} q_{1}}}}  \mathop{ \mathop{\sum \ \sum \ \  \ \sum}_{m,m^{\prime} \sim M_{0} \ 0 \neq n_2 \ll N_2 }}_{\substack{ n_{1} q_{2}^{\prime} d_{2}^{\prime}\mp m n_{2}P_1\bar{P_2} \, \equiv \,  0 \, {\rm mod} \, d_{2} \\  n_{1} q_{2} d_{2}\pm m^{\prime} n_{2}P_1\bar{P_2} \, \equiv \,  0 \, {\rm mod} \, d_{2}^{\prime}}} (m,n_{1}). 
		\end{align}
		Let's first assume that $Cn_1/q_1 \ll M_0$.  In this case, we  count the number of $m$  in the above expression  as follows:
		\begin{align*}
			\sum_{\substack{m \sim M_{0} \\ n_{1} q_{2}^{\prime} d_{2}^{\prime}\mp m n_{2}P_1\bar{P_2} \, \equiv \,  0 \, {\rm mod} \, d_{2}}} (m,n_{1}) & = \sum_{\ell \mid n_{1}} \ell \,  \sum_{\substack{m \sim M_{0}/\ell \\ n_{1} q_{2}^{\prime} d_{2}^{\prime} \bar{\ell}\mp m n_{2}P_1\bar{P_2} \, \equiv \,  0 \, {\rm mod} \, d_{2}}} 1  \\
			& \ll (d_{2},n_{2}) \, \left(n_{1}+\frac{M_{0}}{d_{2}}\right) \notag.
		\end{align*} 
		In the above estimate we have used the fact $(d_2,n_1)=1$. Counting the number of $m^\prime$ in a similar fashion we get that  the $m$-sum  and $m^{\prime}$-sum in \eqref{omega nonzero}  is dominated by 
		$$ (d_{2}^{\prime}, n_{1} q_{2} d_{2}) \, (d_{2}, n_{2}) \left(n_{1}+\frac{M_{0}}{d_{2}}\right) \left(1+\frac{M_{0}}{d_{2}^{\prime}}\right).$$ 
		Now substituting the above bound in \eqref{omega nonzero}, we arrive at
		\begin{align*}
			\frac{q_{1}^2 N_0rC^3} {n_{1}^3M_0^{1/2}Q^3} \mathop{\sum \sum}_{d_{2}, d_{2}^{\prime} \ll \frac{C}{q_1}} d_{2} d_{2}^{\prime} \, \mathop{\sum \sum }_{\substack{q_{2} \sim \frac{C}{d_{2}q_{1}} \\ q_{2}^{\prime} \sim \frac{C}{d_{2}^{\prime} q_{1}}}} \sum_{0<|n_2|\ll N_2}(d_{2}^{\prime}, n_{1} q_{2} d_{2})  (d_{2}, n_{2}) \left(n_{1}+\frac{M_{0}}{d_{2}}\right) \left(1+\frac{M_{0}}{d_{2}^{\prime}}\right).
		\end{align*}
		Now summing  over $n_2$ and $q_2^{\prime}$, we get the following expression:
		\begin{align*}
			 \frac{q_{1} N_0rN_2C^4} {n_{1}^3M_0^{1/2}Q^3 } \mathop{\sum \sum}_{d_{2}, d_{2}^{\prime} \ll C/q_{1} } d_{2}  \, \mathop{ \sum }_{\substack{q_{2} \sim \frac{C}{d_{2}q_{1}}  }} (d_{2}^{\prime}, n_{1} q_{2} d_{2}) \,  \left(n_{1}+\frac{M_{0}}{d_{2}}\right) \left(1+\frac{M_{0}}{d_{2}^{\prime}}\right).
		\end{align*}
		Next we  sum over $d_2^{\prime}$ to  arrive at 
		\begin{align*}
			 \frac{q_{1} N_0rN_2C^4} {n_{1}^3M_0^{1/2}Q^3 } \mathop{\sum }_{d_{2} \ll C/q_{1} } d_{2}  \, \mathop{ \sum }_{\substack{q_{2} \sim \frac{C}{d_{2}q_{1}}  }}    \left(n_{1}+\frac{M_{0}}{d_{2}}\right) \left(\frac{C}{q_1}+M_0\right).
		\end{align*}
		Finally executing the remaining sums,  we get  
		\begin{align*}
			\Omega_{\neq 0} &\ll \frac{q_{1} N_0rN_2C^4} {n_{1}^3M_0^{1/2}Q^3 } \frac{C}{q_1} \left(\frac{Cn_1}{q_1}+M_0\right) \left(\frac{C}{q_1}+M_0\right) \\
			&\ll \frac{ rC^5} {n_{1}^3M_0^{1/2}Q^3 }\frac{CQn_1r}{q_1}\left(\frac{Cn_1}{q_1}+M_0\right) \left(\frac{C}{q_1}+M_0\right) \\
			&\ll \frac{ r^2C^6QM_0^2} {Q^3M_0^{1/2} }\left(\frac{1}{n_1^2q_1}\right)\ll\frac{ r^2C^5Q^2M_0^2} {Q^3M_0^{1/2} }\left(\frac{1}{n_1^2q_1}\right) .
		\end{align*}
		Upon substituting this bound in place of $\Omega$ in \eqref{S(N) after cauchy}, we arrive at 
		\begin{align*}
			&\sup_{C }\frac{N^{17/12}}{P_2^{1/4}P_1^{1/6}Qr^{2/3}C^{5/2}}\frac{ rC^{5/2}M_0^{3/4}} {\sqrt{Q}}\sum_{\pm}\sum_{\frac{n_1}{(n_1,r)}\ll C}n_1^{1/3}\Theta^{1/2}\sum_{\frac{n_1}{(n_1,r)}|q_1|(n_1r)^\infty}\frac{1}{\sqrt{n_1^2q_1}}.
		\end{align*}
		Note that  (for details see \cite{munshi12})
		\begin{align*}
			\sum_{\frac{n_1}{(n_1,r)}\ll C}n_1^{1/3}\Theta^{1/2}\sum_{\frac{n_1}{(n_1,r)}|q_1|(n_1r)^\infty}\frac{1}{\sqrt{n_1^2q_1}} \ll 	\sum_{n_1\ll Cr}\frac{(n_1,r)^{1/2}}{n_1^{7/6}}\Theta^{1/2} \ll N_0^{1/6}.
		\end{align*}
		On plugging in  this estimate, we get
		\begin{align*}
		 \sup_{C }\frac{N^{17/12}}{P_2^{1/4}P_1^{1/6}Qr^{2/3}C^{5/2}}\frac{ rC^{5/2}M_0^{3/4}} {\sqrt{Q}}N_0^{1/6}\ll \sqrt{r}N^{3/4}\sqrt{P_2}.
		\end{align*}
		Next we consider the case where $Cn_1/q_1\gg M_0$. Here our count for $m$ modulo $d_2$ is not precise and so we need to adopt a different strategy for counting.  
 We  consider the first congruence relation  in \eqref{omega nonzero}
		$$n_{1} q_{2}^{\prime} d_{2}^{\prime}\mp m n_{2}P_1\bar{P_2} \, \equiv \,  0 \, {\rm mod} \, d_{2}.$$
		Note that 
		$$n_{1} q_{2}^{\prime} d_{2}^{\prime}P_2\mp m n_{2}P_1\ll CP_2n_1/q_1+M_0{N_2}P_1 \ll CP_2n_1/q_1+CP_2n_1/q_1 \ll CP_2n_1/q_1 .$$
		Let  
		\begin{align}\label{h variable}
			n_{1} q_{2}^{\prime} d_{2}^{\prime}P_2 -m n_{2}P_1 = h d_{2}, \ \ \ \text{with} \ \ h \ll P_2n_1.
		\end{align}
		Similarly, we  write the second congruence relation  as 
		\begin{align}\label{h prime}
			n_{1} q_{2} d_{2}P_2+ m^{\prime} n_{2}P_1 = h^\prime d_{2}^\prime, \ \ \ \text{with} \ \ h^\prime \ll P_2n_1.
		\end{align}
		Using this congruence, we see that the number of $d_2^\prime$ is given by $O((d_2,h^\prime))$. Next we   multiply $h^\prime$ and $P_2q_2^\prime n_1$ into \eqref{h variable} and \eqref{h prime} respectively to arrive at the following equation:
		\begin{align}\label{removed d2 prime}
			m n_{2}P_1h^\prime + hh^\prime d_{2}=	n_{1}^2 q_{2}q_2^\prime d_{2}P_2^2+P_2q_2^\prime n_1 m^{\prime} n_{2}P_1.
		\end{align}
		We now rearrange the above equation as follows:
		$$P_2q_2^\prime n_1 m^{\prime}-mh^\prime=\frac{(hh^\prime -	n_{1}^2 q_{2}q_2^\prime P_2^2)d_2}{P_1n_2}:=\frac{\xi}{P_1n_2}.$$
		Reducing this equation modulo $h^\prime$, the number of $m^\prime$ turns out to be $$O\left((P_2q_2^\prime n_1,h^\prime)\left(1+\frac{P_2}{h^\prime}\right)\right).$$
		Thus we arrive  at the following bound for $\Omega$:
		\begin{align*}
			\frac{q_{1}^2 N_0rC^3} {n_{1}^3M_0^{1/2}Q^3} \mathop{\sum }_{d_{2} \sim \frac{C}{q_1} } \frac{C^2}{q_1^2} \mathop{\sum \sum }_{\substack{q_{2} \sim C^\epsilon \\ q_{2}^{\prime} \sim C^\epsilon}}  \mathop{ \mathop{\sum \ \sum \ \  \ \sum \sum}_{h,\, h^\prime \ll  P_2n_1 \ m \sim M_0, \, n_2 \ll N_2 }}_{\substack{ \xi \, \equiv \,  0 \,  \mathrm{ mod} \,P_1n_2  \\  mh^\prime -\xi/P_1n_2\, \equiv \,  0 \, \mathrm{ mod} \, P_2 }} (m,n_{1})(d_2,h^\prime)(P_2q_2^\prime n_1,h^\prime)\left(1+\frac{P_2}{h^\prime}\right).
		\end{align*}
		
		Next we count the number of $m$ to get 
		\begin{align*}
			\sum_{\substack{m \sim M_{0} \\  mh^\prime -\xi/P_1n_2\, \equiv \,  0 \, \mathrm{ mod} \, P_2}} (m,n_{1}) &= \sum_{\ell \mid n_{1}} \ell \,  \sum_{\substack{m \sim M_{0}/\ell \\  mh^\prime -\xi/\ell P_1n_2\, \equiv \,  0 \, \mathrm{ mod} \, P_2}} 1  
			&\ll \sum_{\ell \mid n_{1}} \ell \notag.
		\end{align*}
	Also given any $\xi$ (necessarily nonzero) the congruence
	$$\xi=(hh^\prime -	n_{1}^2 q_{2}q_2^\prime P_2^2)d_2 \ \equiv \ 0 \ \mathrm{mod}\ n_2,$$
	implies that there are $O(N^\varepsilon)$ many $n_2$.
We are left with the following expression:
		\begin{align*}
			\Omega_{\neq 0}  \ll \frac{ N_0rC^5} {n_{1}^3M_0^{1/2}Q^3} \sum_{\ell |n_1}\ell \mathop{\sum }_{d_{2} \sim \frac{C}{q_1} }  \mathop{\sum \sum }_{\substack{q_{2} \sim C^\epsilon \\ q_{2}^{\prime} \sim C^\epsilon}}  \mathop{ \mathop{\sum \ \sum }_{h,\ h^\prime \ll  P_2n_1  }}_{\substack{ \xi \, \equiv \,  0 \,  \mathrm{ mod} \,P_1\ell}} (d_2,h^\prime)(P_2q_2^\prime n_1,h^\prime)\left(1+\frac{P_2}{h^\prime}\right).
		\end{align*}
		We now consider the congruence 
		$$\xi=(hh^\prime -	n_{1}^2 q_{2}q_2^\prime P_2^2)d_2 \ \equiv \ 0 \ \mathrm{mod}\ P_1\ell.
		$$
		Let's first assume that $d_2 \ \equiv \ 0 \ \mathrm{mod} \ P_1$. Then first counting the number of $d_2$ followed by $h$ and $h^\prime$,  we see that  the number of tuples $(h,h^\prime, d_2)$ is given by 
		$O(\frac{P_2^2n_1^2C}{P_1q_1\ell}).$
		Lastly executing the sum over $\ell$, we arrive at
		\begin{align*}
		\frac{ N_0rC^5} {n_{1}^2M_0^{1/2}Q^3}\frac{P_2^2n_1^2C}{P_1q_1} .
		\end{align*}
		Now let $(d_2,P_1\ell)=1$. Then we have 
		$$hh^\prime -	n_{1}^2 q_{2}q_2^\prime P_2^2 \ \equiv \ 0 \ \mathrm{mod}\ P_1\ell,$$
		from which the number of $h$ turns out to be $P_2n_1/P_1\ell$. Next counting the number of $d_2$ followed by number of $h^\prime$, we see that the number of tuples $(h,h^\prime, d_2)$ is given by 
	$O(\frac{P_2^2n_1^2C}{P_1q_1\ell}).$ 	Hence, in this case also, we get the same bound. Thus  we conclude that
	\begin{align*}
		\Omega_{\neq 0}\ll \frac{ N_0rC^5} {n_{1}^2M_0^{1/2}Q^3}\frac{P_2^2n_1^2C}{P_1q_1}.
	\end{align*}
Upon substituting this bound in \eqref{S(N) after cauchy}, we arrive at 
		\begin{align*}
			&\sup_{C \ll Q}\frac{N^{17/12}}{P_2^{1/4}P_1^{1/6}Qr^{2/3}C^{5/2}}\left(\frac{ N_0rC^5} {Q^3M_0^{1/2}}\frac{P_2^2C}{P_1}\right)^{1/2} \sum_{\pm}\sum_{\frac{n_1}{(n_1,r)}\ll C}n_1^{1/3}\Theta^{1/2}\sum_{\frac{n_1}{(n_1,r)}|q_1|(n_1r)^\infty}\frac{1}{\sqrt{n_1q_1}} \\
			&\ll \frac{N^{17/12}}{P_2^{1/4}P_1^{1/6}Qr^{2/3}}\left(\frac{ N_0r} {Q^3M_0^{1/2}}\frac{P_2^2Q}{P_1}\right)^{1/2} \sum_{\frac{n_1}{(n_1,r)}\ll C}\frac{\Theta^{1/2}}{n_1^{2/3}}(n_1,r)^{1/2}.
		\end{align*}
		Now following the argument of \cite{munshi12},  we conclude that 
		\begin{align*}
			\sum_{\frac{n_1}{(n_1,r)}\ll C}\frac{\Theta^{1/2}}{n_1^{2/3}}(n_1,r)^{1/2} \ll N_0^{1/6} \sum_{\frac{n_1}{(n_1,r)}\ll C}\frac{(n_1,r)^{1/2}}{n_1} \ll N_0^{1/6}.
		\end{align*}
			Hence the contribution of the non-zero frequency to $S_r(N)$ is dominated by
		\begin{align*}
		  \frac{N^{17/12}}{P_2^{1/4}P_1^{1/6}Qr^{2/3}}\left(\frac{ N_0r} {Q^3M_0^{1/2}}\frac{P_2^2Q}{P_1}\right)^{1/2} N_0^{1/6}\ll \sqrt{r} N^{3/4}\sqrt{P_2}.
		\end{align*}
	\end{proof}
	\section{Conclusion}
	Finally, plugging bounds from Lemma \ref{SN for zero frequency} and Lemma \ref{SN for non-zero frequency} into Lemma \ref{AFE}, we get
	\begin{align*}
		L \left( {1}/{2}, F \times f \right) &\ll_{\epsilon} \mathcal{Q}^{\epsilon}\sum_{r \leq \mathcal{Q}^{(1+2\epsilon)/4}}\frac{1}{r} \sup_{ N\leq \frac{\mathcal{Q}^{1/2+\epsilon}}{r^2}} \sqrt{r}N^{1/4}(\sqrt{P_1}+\sqrt{P_2})\\
		& \ll \sum_{r \leq \mathcal{Q}^{(1+2\epsilon)/4}}\frac{1}{r} \mathcal{Q}^{1/8+\epsilon}(\sqrt{P_1}+\sqrt{P_2}) \\
		& \ll \mathcal{Q}^{1/8+\epsilon}(\sqrt{P_1}+\sqrt{P_2}) \ll \mathcal{Q}^{1/4+\epsilon}\left( \frac{P_1^{1/4}}{P_2^{3/8}}+\frac{P_2^{1/8}}{P_1^{1/4}}\right).
	\end{align*}
This establishes Theorem~\ref{main theorem}.
	

\end{document}